\documentclass{amsart}

\usepackage[centertags]{amsmath} 
\usepackage{amsfonts,amssymb,amsthm,amscd,latexsym,bbm,stmaryrd,mathabx,cmll,enumitem}
\usepackage[mathscr]{euscript}
\usepackage[all]{xy}
\usepackage[utf8]{inputenc}

\numberwithin{equation}{section}        
\newtheorem {theorem}[equation]{Theorem}
\newtheorem {corollary}[equation]{Corollary}
\newtheorem {proposition}[equation]{Proposition}
\newtheorem {lemma}[equation]{Lemma}
\theoremstyle{definition}
\newtheorem {definition}[equation]{Definition}

\newtheorem {remark}[equation]{Remark}

\newcommand{\co}{\colon\thinspace}%
\newcommand{\bu}{\bullet}%
\newcommand{\mc}[1]{\ensuremath{\mathcal{#1}}}%
\newcommand{\mb}[1]{\mathbbm{#1}}%
\newcommand{\hofib}{\operatorname*{hofib\,}}%
\newcommand{\coprodnil}[1]{\stackrel{n}{\bigvee}}%
\newcommand{\id}{\operatorname{id}}%
\newcommand{\ul}[1]{\underline{#1}}%
\newcommand{\Sp}{\ensuremath{{\rm Sp}}}%
\newcommand{\Fr}{\ensuremath{{\rm Fr}}}%
\newcommand{\Gr}{\ensuremath{{\rm Gr}}}%
\newcommand{\nil}[1]{\ensuremath{{#1}{\rm -nil}}}
\newcommand{\Sfinp}{\ensuremath{\mc{S}_*^{{\rm fin}}}}%
\newcommand{\dgrm}[1]{\ensuremath{\smash{\underset{\widetilde{\hphantom{#1}}}{#1}} \mathstrut}}%

\usepackage{hyperref} 

\begin{document}

\title  [Homotopy nilpotent groups and their associated functors]{Homotopy nilpotent groups and their associated functors}
\author{Georg Biedermann}

\address{Georg Biedermann, LAGA - Institut Galil\'ee, Universit\'e Paris 13, 99 avenue Jean Baptiste Clément, 93430 Villetaneuse}

\email{biedermann@univ-paris13.fr}

\subjclass{55P65, 55P35}

\keywords{homotopy nilpotent groups, Goodwillie calculus, loop spaces, sifted colimits}

\date{\today}

\begin{abstract}
To every homotopy $n$-nilpotent group, defined in earlier work by Dwyer and the author, we associate an endofunctor of pointed spaces and prove that it is looped and $n$-excisive. 
As a tool we prove that $\Omega P_n(\id)$ commutes with sifted colimits of connected spaces.
\end{abstract}

\maketitle

\section{Introduction}

Homotopy nilpotent groups were introduced in \cite{Biedermann-Dwyer:nilpotent} by W. G. Dwyer and the author in order to understand the relation between the Goodwillie tower of the identity and (the classifying spaces of) the lower central series of the Kan loop group on connected spaces. It turns out that they interpolate between loop spaces $n=\infty$ and infinite loop spaces $n=1$, but in reverse. The simplicial algebraic theory whose homotopy algebras are the homotopy $n$-nilpotent groups consists of finite products of the functor $\Omega P_n(\id)$. The $k$-ary operations of this theory are $\Omega P_n(\id)\Sigma \{1,\hdots, k\}_+$ where $+$ denotes a disjoint basepoint. In~\cite{Biedermann-Dwyer:nilpotent} we proved that the ordinary algebraic theory obtained by applying $\pi_0$ is the theory of $n$-nilpotent groups; hence the name. 

In the recent paper~\cite{costoya-scherer-viruel:torus-theorem} Costoya, Scherer, and Viruel compare our nilpotence degree for loop spaces to classical invariants like Berstein-Ganea nilpotence~\cite{Berstein-Ganea} and Ganea's inductive cocategory~\cite{Ganea:(co)category}. They obtain results for $p$-compact groups and $p$-Noetherian groups. In particular, they obtain an analogue of Hubbuck's torus theorem~\cite{Hubbuck:torus} for our notion of nilpotence. 

In a crucial step they use that one can associate to a homotopy $n$-nilpotent group a functor from finite pointed spaces to pointed spaces and that this functor is looped and $n$-excisive in the sense of Goodwillie~\cite{Goo:calc2, Goo:calc3}. This was announced in~\cite{Biedermann-Dwyer:nilpotent} but never proved. We are now going to fill the gap with Proposition~\ref{X is n-excisive}. Theorem~\ref{values-n-excisive} then states that a pointed space has the structure of a homotopy $n$-nilpotent group if and only if it is the value of a looped $n$-excisive functor. Another claim that was announced but not proved so far.

{\bf Acknowledgement:}
This note goes way back to discussions with Bill Dwyer when we were working on~\cite{Biedermann-Dwyer:nilpotent} and I owe him many thanks and credits. Similarly, I am indebted to Andr\'e Joyal for many discussions. I received comments from Ji\v{r}\'i Rosick\'{y} and J{\'e}r{\^o}me Scherer. In particular, I would like to thank Cristina Costoya, J{\'e}r{\^o}me Scherer, and Antonio Viruel for writing their paper~\cite{costoya-scherer-viruel:torus-theorem} and prompting me to revisit homotopy nilpotent groups.

This project has received funding from the
European Union’s Horizon 2020 research and innovation programme under
Marie Sk\l odowska-Curie grant agreement No 661067.

\section{Homotopy nilpotent groups}

Before we recall the definition of homotopy-$n$-nilpotent groups let us agree on some notation. 
We denote by $\mc{S}_*$ the category of pointed simplicial sets and by \mc{F} the category of $\mc{S}_*$-enriched functors from finite pointed simplicial sets to $\mc{S}_*$. All such functors are strictly reduced in the sense that $F(\ast)\cong\ast$. The category \mc{F} can be given the injective model structure where weak equivalences and cofibrations are given objectwise. In this model structure all object in \mc{F} are cofibrant.
Let $(-)^{\rm inj}$ be an objectwise fibrant replacement functor in the injective model structure.
For $k\ge 0$ let $k_+$ be the set $\{1,\hdots,k\}$ with an added basepoint $+$. 

In~\cite{Biedermann-Dwyer:nilpotent} homotopy $n$-nilpotent groups were introduced as homotopy algebras over a certain simplicial algebraic theory $\mc{G}_n$.
For $1\le n\le\infty$ the simplicial category $\mc{G}_n$ is given by
   $$ \mc{G}_n(k^+)=\Omega\mc{P}_n(k^+)=\prod_{i=1}^k\Omega(P_n(\id))^{\rm inj} $$
viewed as a full subcategory of \mc{F}. Here $P_n(\id)$ refers to the $n$-excisive approximation of the identity functor of $\mc{S}_*$ constructed by Goodwillie~\cite{Goo:calc3}.
The category of $\mc{S}_*$-enriched functors from $\mc{G}_n$ to $\mc{S}_*$ carries a Badzioch model structure in the sense of~\cite{Badzioch:theories}. We will denote this model category by $\Gr^{\nil{n}}$ and call it the {\it category of homotopy $n$-nilpotent groups}. The fibrant objects are the homotopy algebras of the theory $\mc{G}_n$; these are the functors that preserve products up to weak equivalence. We call them {\it homotopy $n$-nilpotent groups}.

The {\it forgetful functor} $u\co\Gr^{\nil{n}}\to\mc{S}_*$ is given by evaluating a homotopy $n$-nilpotent group, viewed as a product-preserving functor from $\mc{G}_n$, at $1_+$. We call 
  $$ u(\ul{X})=\ul{X}(1_+)$$ 
the {\it underlying space} of $\ul{X}$. 
As for all algebraic theories the forgetful functor has a left adjoint: we write 
  $$\Fr_n\co\mc{S}_*\to\Gr^{\nil{n}} $$
for the free homotopy $n$-nilpotent group functor.

Since tensoring homotopy nilpotent groups with pointed spaces will play an important role here we point out that the tensor product of $\mc{S}_*$ over itself, ie. the smash product, has the following description:
  $$ \left|\, s\mapsto \bigvee_{K_s-\{k_0\}} X\, \right|\simeq X\wedge K $$
where $X$ and $K$ are in $\mc{S}_*$, and $k_0\in K_0$ denotes the basepoint (and all its degeneracies).

\begin{definition}
Let $\ul{X}$ be a homotopy $n$-nilpotent group. Let $K$ be in $\mc{S}_*$ with basepoint $k_0\in K_0$. We define a tensor of $\Gr^{\nil{n}}$ over $\mc{S}_*$ in the following way:
   $$ \ul{X}\wedge_n K:= \left|\, s\mapsto \bigvee_{K_s-\{k_0\}}\ul{X}\ \right|.$$
Here $|-|$ denotes the realization of a simplicial object in $\Gr^{\nil{n}}$ and the coproduct is to be taken in $\Gr^{\nil{n}}$.

A cotensor of $\Gr^{\nil{n}}$ over $\mc{S}_*$ is defined objectwise where we view the category $\Gr^{\nil{n}}$ as a full subcategory of \mc{F}, compare~\cite[Section 3]{Biedermann-Dwyer:nilpotent}.
\end{definition}

With tensor and cotensor defined as above it is routine to verify the following
\begin{lemma}
The category $\Gr^{\nil{n}}$ is tensored and cotensored over $\mc{S}_*$. Together with the Badzioch model structure it becomes an $\mc{S}_*$-model category.
\end{lemma}

We gave the definition of $\wedge_n$ in terms of strict colimits to make sense of the previous lemma. It is important that from now we will only consider the {\it derived functor of $\wedge_n$} given by replacing the strict colimit with a homotopy colimit. We keep the notation $\wedge_n$ for this homotopy invariant since the conventional $\wedge_n^L$ seems a bit cumbersome. Similarly, from now on the term ``realization'' and the notation $|-|$ will refer to the {\it derived realization}.

Let $\ul{X}$ be a homotopy $n$-nilpotent group. Let $K$ be a pointed space.
One can compute
   $$ \ul{X}\wedge_{\infty} K\simeq\Omega(B\ul{X}\wedge K) $$
and
   $$ \ul{X}\wedge_{1}K\simeq\Omega^{\infty}\bigl((B^{\infty}\ul{X})\wedge K\bigr) $$
where $B$ ($B^\infty$) denotes the (infinite) delooping. This is proved in Lemma~\ref{assoc-fun-of-free}.

\section{The functor $\Omega P_n(\id)$ and realization}

The goal of this section is to prove the following

\begin{proposition}\label{Omega-Pn-commutes-with-sifted-colim}
The functor $\Omega P_n(\id)$ commutes with sifted homotopy colimits in the category of connected pointed spaces.
\end{proposition}

\begin{proof}
A functor commutes with all sifted homotopy colimits if and only if it commutes all filtered homotopy colimits and all realizations.
Both functors $\Omega$ and $P_n(\id)$ commute with filtered homotopy colimits. Thus, it suffices to check that $\Omega P_n(\id)$ commutes with (derived) realization of degreewise connected simplicial spaces. This is done in Lemma~\ref{lem:OPSreal}.
\end{proof}

As immediate consequence we can give an alternative description of the free homotopy $n$-nilpotent group of pointed space.

\begin{proposition}
For a pointed space $K$ we have: $\Fr_n(K)\simeq\Omega P_n(\id)\Sigma K$.
\end{proposition}

\begin{proof}
It was proved in~\cite[Cor. 5.7]{Biedermann-Dwyer:nilpotent} that
  $$ \Fr_n(k_+)\simeq\Omega P_n(\id)\left(\bigvee_k S^1\right)=\Omega P_n(\id)\Sigma (k_+) . $$
Thus, the claim holds for finite discrete pointed spaces. Any other pointed space is the sifted homotopy colimit of finite discrete pointed spaces. Now Lemma~\ref{Omega-Pn-commutes-with-sifted-colim} applies and finishes the proof since $\Sigma (k_+)$ is connected.
\end{proof}

The rest of this section is devoted to the proof of Lemma~\ref{lem:OPSreal}. It is mainly an induction up the Goodwillie tower. We start by preparing its ingredients.

\begin{lemma}\label{lem:realization-spectra}
The realization functor $|-|\co s\Sp\to\Sp$ from simplicial spectra to spectra commutes with finite homotopy limits up to weak equivalence.
\end{lemma}

\begin{proof}
Realization commutes with homotopy pullbacks by stability. It also preserves the terminal object. Hence, it commutes with all finite homotopy limits.
\end{proof}

\begin{corollary}\label{cor:realization-P-n}
Let $F_\bullet$ be a simplicial object of functors to spectra. Then the canonical maps 
\begin{align*} 
   |P_nF_\bu|\xrightarrow{\simeq} P_n|F_\bu| \ \ \text{and} \ \ |D_nF_\bu|\xrightarrow{\simeq} D_n|F_\bu| 
\end{align*}
are weak equivalences.
\end{corollary}

\begin{proof}
Homotopy limits and colimits of functors to spectra are computed objectwise. In general, realization commutes with homotopy colimits. But in spectra it also commutes with finite homotopy limits by Lemma~\ref{lem:realization-spectra}. This implies that $P_n$ commutes with realization. The equivalence for $D_n$ now follows directly because a homotopy colimit of $n$-reduced functors is still $n$-reduced. 
\end{proof}

\begin{lemma}\label{lem:omega-inf-realization}
Let $Y$ be a simplicial spectrum that is degreewise connective. Then
  $$ \Omega^{\infty}|Y|\simeq |\Omega^{\infty}Y|.  $$
\end{lemma}

\begin{proof}
The functor $\Omega^\infty\co\Sp\to\mc{S}_*$ factors into the connective cover functor from spectra to connective spectra, followed by the equivalence, usually also denoted by $\Omega^\infty$, of connective spectra with infinite loop spaces, and finally followed by the forgetful functor from infinite loop spaces to pointed spaces. Since we restrict ourselves to connective spectra the first functor becomes an equivalence. The third functor is a forgetful functor from an algebraic theory, the algebraic theory of infinite loop spaces, and hence commutes with sifted colimits. 
\end{proof}

We are going to use a lemma proved by Rezk. We state it here in a weaker form, the form that we are actually going to apply. 

\begin{lemma}[Rezk, Cor. 5.8 \cite{Rezk:hate-pi-star-Kan}]\label{lem:everyone-hates-pi-star-Kan}
Let $p\co E\to B$ be a map of $H$-group objects in simplicial spaces such that $\pi_0(p)\co \pi_0E\to\pi_0B$ is surjective. Then:
  $$ |\hofib p|\simeq \hofib|p|. $$
\end{lemma}

The input for Rezk's result comes from the following
\begin{lemma}\label{lem:P-nSigmaconnected}
Let $K$ in $\Sfinp$ be connected. Then, for all $n\in\mb{N}$,
\begin{enumerate}
  \item[{\rm (1)}]
the space $P_n(\id)(K)$ is connected, and 
  \item[{\rm (2)}]
the map $q_n\co P_n(\id)(K)\to P_{n-1}(\id)(K)$ induces a surjection on $\pi_1$.
\end{enumerate}
\end{lemma}

\begin{proof}
We know by computations of Johnson~\cite{Johnson:thesis} and Arone-Mahowald~\cite{AroMah:id} that the spectrum $\partial^n\id$ is $(-n)$-connected.
Thus, the spectrum $\partial^n\id\wedge K^{\wedge n}$ is connected because $K$ is connected.
The homotopy orbit by the action of $\Sigma_n$ does not destroy the connectivity. This can be seen as in~\cite[Lem. 7.3]{Biedermann-Dwyer:nilpotent} by using rational homology, the Serre spectral sequence, and the fact that the higher homology of $\Sigma_n$ with coefficients in a rational vector space vanishes.

The other two statements now follow from the long exact sequence associated to the homotopy fiber sequence
 $$ D_n(\id)(K)\to P_n(\id)(K)\to P_{n-1}(\id)(K) $$
for each $n\ge 1$. 
\end{proof}

\begin{lemma}\label{lem:OPSreal}
Let $K_\bu$ be a degreewise connected simplicial space. Then, for all $1\le n\le\infty$, the canonical map
  $$ |\Omega P_n(\id) (K_\bu)|\to \Omega P_n(\id)|K_\bu| $$
is a weak equivalence.
\end{lemma}

\begin{proof}
First note that for $n=\infty$ the statement says that $|\Omega K_\bu|\simeq\Omega |K_\bu|$ if $K_s$ is connected for all $s$. This is well-known, see e.g. Theorem 12.3~\cite{MR0420610}.
Similarly, the case $n=1$, that concerns the functor $\Omega P_1(\id)\simeq \Omega^\infty\Sigma^\infty$, holds by Lemma~\ref{lem:omega-inf-realization}.

Now suppose $1\le n<\infty$. Let us consider the Goodwillie tower
  $$ \Omega P_n(\id)(K_\bu) \xrightarrow{q_{n-1}} \Omega P_{n-1}(\id)(K_\bu) \to\hdots $$
For each $1\le k\le n$ the homotopy fiber $\Omega D_k(\id)$ is an endofunctor of $\mc{S}_*$ that factors through spectra. In fact, it factors through connective (ie.~$(-1)$-connected) spectra by Lemma~\ref{lem:P-nSigmaconnected} when restricted to connected spaces.  
By Corollary~\ref{cor:realization-P-n} the corresponding functor to spectra commutes with realization of degreewise connected simplicial spaces.  
By Lemma~\ref{lem:omega-inf-realization} the functor $\Omega D_k(\id)$ then commutes itself with realization of degreewise connected simplicial spaces.  

Now we are going to prove the lemma by induction on $n$.
The case $n=1$ was already dealt with above.
The following sequence 
  $$ |\Omega D_k(\id)(K_\bu)| \to |\Omega P_k(\id)(K_\bu)| \xrightarrow{q_k} |\Omega P_{k-1}(\id)(K_\bu)| $$
is a homotopy fiber sequence for all $1\le k\le n$. This follows from Rezk's Lemma~\ref{lem:everyone-hates-pi-star-Kan}: the map $q_k$ is a map of $H$-group objects in simplicial spaces which is surjective on $\pi_0$ due to Lemma~\ref{lem:P-nSigmaconnected}.
The induction step works because, for all $1\le k\le n$, the comparison map 
  $$ |\Omega P_k(\id) (K_\bu)|\to \Omega P_k(\id)|K_\bu| $$
is the middle part of a morphism of homotopy fiber sequences whose two other parts are weak equivalences.
\end{proof}

\section{Functors associated to homotopy $n$-nilpotent groups}

\begin{definition}\label{def:associated-functor}
Let $\ul{X}$ be a homotopy $n$-nilpotent group. We associate to it a functor
   $$ \dgrm{X}\co\Sfinp\to\Gr^{\nil{n}} $$
given by
   $$ \dgrm{X}(K):=\ul{X}\wedge_n K.$$
by composing it with the forgetful functor $u\co\Gr^{\nil{n}}\to\mc{S}_*$ it can be viewed as a functor with values in pointed spaces and we denote it by the same letter.
\end{definition}

\begin{remark}\label{assoc-fun-pres-colim}
The functor 
  $$\Gr^{\nil{n}}\to\mc{F}\ ,\ \ \ul{X}\mapsto\dgrm{X} $$
preserves all (homotopy) colimits.
\end{remark}

\begin{lemma}\label{assoc-fun-of-free}
For two pointed space $X$ and $K$ and each $1\le n\le\infty$ we have
  $$\Fr_n(X)\wedge_n K\simeq \Omega P_n(\id)\Sigma (X\wedge K). $$
\end{lemma}

\begin{proof}
Note that the free functor preserves (homotopy) colimits: 
$$\begin{array}{rl} \Fr_n(X)\wedge_n K & = \left|\,\bigvee_{K-\{k_0\}}\Fr_n(X)\,\right| \cong \left|\,\Fr_n(\bigvee_{K-\{k_0\}} X)\,\right| \cong \Fr_n\left|\,(\bigvee_{K-\{k_0\}} X)\,\right| \\ \\
   &\simeq \Omega P_n(\id)\Sigma (X\wedge K) . 
\end{array}$$
\end{proof}

\begin{corollary}\label{free-groups-yields-excisive-functors}
For all $1\le n<\infty$ the functor associated to a free homotopy $n$-nilpotent group is $n$-excisive.
\end{corollary}

\begin{proof}
The functor $\Omega P_n(\id)(\Sigma X\wedge K)$ is clearly $n$-excisive in $K$.
\end{proof}

In the next proof we need to consider the two-sided bar construction associated to the adjoint pair $(\Fr_{n},u)$. This is a simplicial object in $\Gr^{\nil{n}}$ given by $(\Fr_nu)^{s+1}(X)$ in degree $s$. The standard notation for this object is $B^\bu(\Fr_n, u\Fr_n,u)$. We are going to use the following notation
  $$ B^\bu(\Fr_n, u\Fr_n,u\ul{X})(K)=\Fr_n\bigl((u\Fr_n)^{s}u(\ul{X})\bigr)\wedge_n K=\Omega P_n(\id)\Sigma\bigl(((u\Fr_n)^{s}u\ul{X})\wedge K\bigr) $$
for the associated simplicial object of functors to $\mc{S}_*$ where $s$ denotes the simplicial degree and $K$ is the input variable for the functor. This is inconsistent because using our earlier notation we would need to write an underline tilde. We hope the reader agrees that this slight abuse of notation is convenient.

\begin{proposition}\label{X is n-excisive}
The functor \dgrm{X} associated to a homotopy $n$-nilpotent group \ul{X} is of the form $\Omega F$ where $F$ in \mc{F} is $n$-excisive.
\end{proposition}

\begin{proof}
For the special case of a free homotopy $n$-nilpotent group $\dgrm{X}=\Fr_n(X)$ for some pointed space $X$ this was shown in Corollary~\ref{free-groups-yields-excisive-functors}. 

Now every homotopy $n$-nilpotent group $\ul{X}$ is weakly equivalent to the realization of its two-sided bar construction. Following Remark~\ref{assoc-fun-pres-colim} this process commutes with homotopy colimits. Hence, the associated functor $\dgrm{X}$ is weakly equivalent to the realization of the simplicial object of associated functors of the bar construction: for all $K$ in $\Sfinp$
  $$\dgrm{X}(K)\simeq |\, B^\bu(\Fr_n, u\Fr_n,u\ul{X})(K)\, | .$$

Before we continue let us note that the sequence with $(u\Fr_n)^{s}u\ul{X}$ in degree $s$ does not form a simplicial object and Lemma~\ref{lem:OPSreal} does not apply. The idea of the following argument, however, is the same.

It is now important that we are considering the bar construction as a simplicial object of functors with values in pointed space, not $\Gr^{\nil{n}}$. This affords us the existence of a map 
  $$\xymatrix{ B^{s}(\Fr_n,u\Fr_n,u\ul{X})(K)\ar@{=}[r]\ar[d] & \Omega P_n(\id)\Sigma\bigl(((u\Fr_n)^{s}u\ul{X})\wedge K\bigr) \ar[d] \\
 B^{s}(\Omega P_{n-1}(\id)\Sigma,u\Fr_n,u\ul{X})(K) \ar@{=}[r] & \Omega P_{n-1}(\id)\Sigma\bigl(((u\Fr_n)^{s}u\ul{X})\wedge K\bigr) } $$
of simplicial objects in the functor category \mc{F}, because there is a map
  $$ \Omega P_{n-1}(\id)\Sigma (u\ul{X})\to u\ul{X} $$
of pointed spaces that supplies the missing face maps. Of course, this continues to hold for excisive degrees smaller than $n-1$. Thus, we can consider the Goodwillie tower
  $$ \Omega P_n(\id)\Sigma\bigl(((u\Fr_n)^{s}u\ul{X})\wedge K\bigr) \to \Omega P_{n-1}(\id)\Sigma\bigl(((u\Fr_n)^{s}u\ul{X})\wedge K\bigr) \to \hdots $$
which is a tower of simplicial objects.
For each $1\le k\le n$ the layer of this tower is
  $$ \Omega D_k(\id)\bigl(((u\Fr_n)^su\ul{X})\wedge K\bigr), $$ 
again viewed as a simplicial functor of $K$. All of these layers factor through spectra.  By Corollary~\ref{cor:realization-P-n} and Lemma~\ref{lem:omega-inf-realization} the realization of the layers is $k$-excisive. 

The induction starts with $n=1$ and the fact that $\Omega P_1(\id)\Sigma=\Omega^\infty\Sigma^\infty$ commutes with realization.
The induction step follows since the diagram of realizations
  $$\xymatrix{ |\Omega D_k(\id)\bigl(((u\Fr_n)^su\ul{X})\wedge K\bigr)| \ar[r]\ar[d] & |\Omega P_k(\id)\bigl(((u\Fr_n)^su\ul{X})\wedge K\bigr)| \ar[d]^{q_k} \\
    \ast \ar[r] & |\Omega P_{k-1}(\id)\bigl(((u\Fr_n)^su\ul{X})\wedge K\bigr)|  } $$
is a homotopy pullback for all $1\le k\le n$ by Rezk's Lemma~\ref{lem:everyone-hates-pi-star-Kan} and Lemma~\ref{lem:P-nSigmaconnected}.
\end{proof}

\begin{theorem}\label{values-n-excisive}
A space $X$ has the structure of a homotopy $n$-nilpotent group if and only if there exists a weak equivalence
    $$ X\simeq \Omega F(S^0) $$
for some $n$-excisive functor $F$.
\end{theorem}

\begin{proof} 
We proved in \cite[9.2]{Biedermann-Dwyer:nilpotent} that every value of a functor $\Omega F$, where $F$ is $n$-excisive, is a homotopy $n$-nilpotent group. 

For the converse, let $\ul{X}$ be a homotopy $n$-nilpotent group. 
For the associated functor \dgrm{X} we obviously have
   $$ \ul{X}\cong\dgrm{X}(S^0). $$
By Proposition~\ref{X is n-excisive} $\dgrm{X}$ is a looped $n$-excisive functor.
\end{proof}


\end{document}